\documentclass[journal]{IEEEtran}

\ifCLASSINFOpdf
\else
\fi

\usepackage{amsthm}
\usepackage{amssymb}
\usepackage{amsmath}
\usepackage{times}
\usepackage{graphicx}
\usepackage{graphics}
\usepackage{subfigure}
\usepackage{multirow}
\usepackage{caption}
\usepackage{makecell}
\usepackage{cite}
\usepackage{color}
\usepackage{booktabs}
\usepackage{epstopdf}
\usepackage{cite}
\usepackage{bm}
\usepackage[noend]{algpseudocode}
\usepackage{algorithmicx,algorithm}

\usepackage{color}
\newtheorem{theorem}    {Theorem}
\newtheorem{definition} {Definition}
\newtheorem{corollary}  {Corollary}
\newtheorem{lemma}      {Lemma}
\newtheorem{remark}     {Remark}

\begin{document}

\title{
Fast consensus of high-order multi-agent systems}

\author{Jiahao~Dai,
        Jing-Wen~Yi$^*~\IEEEmembership{Member,~IEEE,}$,
        Li~Chai,~\IEEEmembership{Member,~IEEE,}

\thanks{This work was supported by the National Natural Science Foundation of China (62173259, 62176192, 61625305 and 61701355).}
\thanks{Jiahao Dai and Jing-Wen Yi are with the Engineering Research Center of Metallurgical Automation and Measurement Technology, Wuhan University of Science and Technology, Wuhan 430081, China. e-mail: daijiahao@wust.edu.cn; yijingwen@wust.edu.cn.}
\thanks{Li Chai is with the College of Control Science and Engineering, Zhejiang University, Hangzhou 310027, China. e-mail: chaili@zju.edu.cn.}}

\markboth{Journal of \LaTeX\ Class Files,~Vol.~xx, No.~x, August~202x}%
{Shell \MakeLowercase{\textit{et al.}}: Bare Demo of IEEEtran.cls for IEEE Journals}

\maketitle

\begin{abstract}
 In this paper, the fast consensus problem of high-order multi-agent systems under undirected topologies is considered.
   The direct link between the consensus convergence rate and the control gains is established.
   An accelerated consensus algorithm based on gradient descent is proposed to optimize the convergence rate.
   By applying the Routh-Hurwitz stability criterion, the lower bound on the convergence rate is derived, and
   explicit control gains are derived as the necessary condition to achieve the optimal convergence rate.
   Moreover,
   a protocol with time-varying control gains is designed to achieve the finite-time consensus.
   Explicit formulas for the time-varying control gains and the final consensus state are given.
   Numerical examples and simulation results are presented to illustrate the obtained theoretical results.
\end{abstract}


\IEEEpeerreviewmaketitle

\section{Introduction}

Consensus is a fundamental problem in distributed coordination, which has been extensively studied in \cite{2013An,2016Recent,dorri2018multi,chen2019control}.
The main purpose of consensus is to design a control protocol, which use the local information between an agent and its neighbors, such that the states of all agents can reach a common value over time.

The convergence rate is an important indicator to evaluate the consensus performance.
There are many methods to accelerate the convergence rate, which can be roughly summarized as:
optimizing the weight matrix \cite{XIAO200465,4627467,7389373},
using the time-varying control \cite{6338354,KIBANGOU201419,6876198}, and
introducing the agent's memory \cite{5411823,8716798,pasolini2020exploiting}.
Recently, Yi et al. \cite{2019Average} gave an explicit formula for the optimal convergence rate of first-order MASs from the perspective of graph signal frequency domain filtering, and Dai et al. \cite{9763023} proposed a general control protocol with memory to accelerate the consensus of first-order MASs.

Most of the above methods to optimize the convergence rate are considered in the first-order system.
However, a broad class of systems have multiple degrees of freedom in practical applications, where the input-output relationship needs to be illustrated by higher-order dynamics \cite{5708233,2011Average,2011Distributed,7054482,8272377}.
Then some researchers explored the convergence rate of higher-order MASs.
Li et al. \cite{LI20111706} studied a consensus algorithm for MASs with double-integrator dynamics, and proved that the finite-time consensus can be achieved by using Lyapunov stability theory.
Under assumptions that the system matrix is controllable and the product of the unstable eigenvalues of the open-loop system matrix has a upper bound, You et al.  \cite{2011Network} provided a lower bound of the optimal convergence rate for high-order discrete-time MASs.
Eichler et al. \cite{2014Closed} proposed a protocol for the consensus of MASs with discrete-time double-integrator dynamics, and derived the optimal control gain by minimizing the largest eigenvalue modulus of the closed-loop system matrix.
Parlangeli et al. \cite{2018Accelerating} proposed a control protocol in high-order continuous-time leader-follower networks, and indicated that
the convergence can be achieved arbitrarily fast by allocating all the eigenvalues of the closed-loop system matrix.

In this paper, the fast consensus problem of high-order MASs is considered.
Control protocols with constant control gains and time-varying control gains are used to achieve the accelerated asymptotic consensus and the finite-time consensus, respectively.
The main contributions of this paper are summarized as follows.
\\(i) The necessary and sufficient condition for high-order MASs to achieve asymptotic consensus is given.
The direct link between the consensus convergence rate and the control gains is established.
\\(ii) The accelerated asymptotic consensus problem is transformed into an optimization problem of the convergence rate, and an accelerated consensus algorithm based on gradient descent is proposed to solve this problem.
By applying the Routh-Hurwitz stability criterion, the lower bound on the convergence rate is given, and
explicit control gains are derived as the necessary condition to achieve the optimal convergence rate.
\\(iii) The explicit formula of time-varying control gains to achieve the finite-time consensus is derived by applying the Cayley-Hamilton theorem.
It is shown that the step to achieve consensus is determined by the system's order and the number of distinct non-zero eigenvalues of Laplacian.

The rest of this paper is organized as follows.
In Section II, the problem statement are presented.
Section III proposes some consensus conditions, designs an accelerated consensus algorithm to accelerate consensus, and derives a lower bound on the convergence rate.
Section IV introduces a time-varying control protocol to achieve the finite-time consensus, and gives explicit formulas for the time-varying control gains.
In Section V, numerical examples are given to verify the theoretical analysis.
Finally, Section VI concludes this paper.

\textbf{Notations:} The notations used in this paper are standard.
The notation $diag\{\,\cdots\}$ denotes a block-diagonal matrix.
$\mathbb{R}^{n}$ and $\mathbb{R}^{m\times n}$ are the sets of column vectors of dimension ${n}$ and matrices of dimension ${m \times n}$ with real
elements, respectively.
${\left\|  \,\cdot\,  \right\|_2}$ denotes the Euclidean norm.
The symbol $\otimes$ stands for Kronecker product.
The symbol $C_p^q=\frac{{p!}}{{q!(p - q)!}}$ denotes the number of $q$-combinations from a given set of $p$ elements.
Without special explanation, $\bm{0}$ and $I$ represent the zero matrix and identity matrix with appropriate dimensions, and $\bm{1}$ denotes the vector of all ones.

\section{Preliminaries}

\subsection{Graph Theory}

The interactions among agents are modeled as an undirected graph $\mathcal{G\!=\!(V,E,A)}$, where $\mathcal{V}\!=\!\{\text{v}_1,\text{v}_2,\cdots,\text{v}_N\}$ presents a set of agents or nodes,
$\mathcal{E \!\subseteq\! V\times V}$ presents a set of edges,
and $\mathcal{A}\!=\![a_{ij}]\!\in\! \mathbb{R}^{N\!\times\! N}$ presents the weighted adjacency matrix.
The adjacency element $a_{ij}\!=\!a_{ji}\!>\!0$ if the edge between node $i$ and $j$ satisfies $e_{ij}\in\mathcal{E}$.
Denote the set of neighbors of node as ${\mathcal{N}_i} \!= \!\left\{ {{\text{v}_j} \in \mathcal{V}:({\text{v}_i},{\text{v}_j}) \in \mathcal{E} } \right\}$.
Define the Laplacian matrix of $\mathcal{G}$ as $\mathcal{L=D-A}$,
where $\mathcal{D}\!=\!diag\{d_{1},\cdots,d_{N}\}$ and ${d_i}\!= \!\sum\nolimits_{j = 1}^N {{a_{ij}}}$.
For a connected graph, all the eigenvalues of $\mathcal{L}$ are real in an ascending order as $0 \!=\! {\lambda _1} \!<\! {\lambda _2} \!\le\!  \cdots  \!\le\! {\lambda _N}$.

\begin{lemma} \cite{2004Consensus}
For any connected undirected graph $\mathcal{G}$, its Laplacian matrix has the following properties.
\\(i) $\mathcal{L}$ has the spectral decomposition ${\mathcal{L}=V\Lambda V^{T}}$, where $\Lambda\!=\!diag\{\lambda_{1},\lambda_{2},\cdots,\lambda_{N}\}$ and $V\!=\![\bm{v}_{1},\bm{v}_{2},\cdots,\bm{v}_{N}]\!\in\!\mathbb{R}^{N\times N}$.
\\(ii) Zero is a single eigenvalue of $\mathcal{L}$, and the corresponding eigenvector is $\bm{v}_{1}\!=\!\frac{1}{\sqrt{N}}\bm{1}$.
\end{lemma}

\subsection{Problem Formulation}

Agents might only be able to interact with their neighbors intermittently rather than continuously because digital signals communicate in discrete time.
A discrete-time high-order MAS containing $N$ agents with order $n \geq 1$ is considered as follows.
\begin{equation}
\begin{array}{*{20}{l}}
{\begin{array}{*{20}{l}}
{x_i^{(1)}(k + 1) = x_i^{(1)}(k) + x_i^{(2)}(k) \cdot \tau ,}\\
{x_i^{(2)}(k + 1) = x_i^{(2)}(k) + x_i^{(3)}(k) \cdot \tau ,}\\
 \;\;\;\;\;\;\;\;\;\;\;\;\;\;\;\;\;\;\;\vdots \\
{x_i^{(n)}(k + 1) = x_i^{(n)}(k) + {u_i}(k) \cdot \tau ,}
\end{array}}&\begin{array}{l}
i = 1,2, \ldots ,N,\\
l = 1,2, \ldots ,n.
\end{array}
\end{array}
\end{equation}
where $x_i^{(l)}(k)\in \mathbb{R}$ represents the $l$-order state of the agent $i$, $u_i(k)\in\mathbb{R}$ is the control input, and $\tau\in\mathbb{R}^{+}$ denotes the sampling period.

Let $\bm{x}_{i}(k)=[x_{i}^{(1)}(k),\cdots,x_{i}^{(n)}(k)]^T$ and rewrite system (1) into a matrix form
\begin{equation}
\begin{array}{*{20}{c}}
  {{{\bm{x}}_i}(k + 1) = A{\bm{x}_i}(k) + B{u_i}(k),}&{i = }
\end{array}1,2, \ldots ,N,
\end{equation}
where
\begin{equation}
A = \left[ {\begin{array}{*{20}{c}}
{\;1\; }&{\; \tau \;} &{\;\;}&{\;\;}\\
{\;\;}&{\; 1\; }& {\; \ddots \;} &{\;\;}\\
{\;\;}&{\;\;}& { \;\ddots \;} &{\; \tau \;} \\
{\;\;}&{\;\;}&{\;\;}&{\; 1 \;}
\end{array}} \right],B = \left[ {\begin{array}{*{20}{c}}
0\\
 \vdots \\
0\\
\tau
\end{array}} \right].
\end{equation}
\begin{remark}
Ma et al. \cite{2010Necessary} studied the consensus problem of high-order MASs, indicating that the state of each agent converges to zero without taking any control when the open-loop system is stable.
It means that studying the consensus of open-loop stable systems is of little significance.
For an unstable open-loop system, it is usually necessary to make some assumptions to achieve consensus.
However, these assumptions make the conclusions obtained conservative.
For example, References \cite{2011Network,LI2018144} limit the range of eigenratio $\lambda_2 / \lambda_N$.
In fact, for an unstable open-loop system,
each agent can use a local controller  ${u_i}(k) = K_i{x_i}(k)$ for pole-placement,
and make the open-loop system marginally stable.
Then the neighbor information can be utilized to achieve consensus.
Therefore, the marginally stable open-loop system  considered in \cite{2011Distributed,5708233,7054482} and this paper is not loss of generality.
Instead, we think it is more suitable for practical applications.
\end{remark}

\begin{definition}
Consider the high-order MAS (1) with arbitrary initial value.
\\(i) Consensus is said to be reached asymptotically if
\[\begin{array}{*{20}{l}}
{\mathop {\lim }\limits_{k \to \infty } \left[ {x_i^{(l)}(k) - x_j^{(l)}(k)} \right] = 0},\,\,\\i,j = 1,2, \cdots ,N ,\,\,l = 1,2, \cdots ,n.
\end{array}\]
\\(ii) Consensus is said to be reached at step $\rm{T}$ if
\[\begin{array}{*{20}{c}}
{x_i^{(l)}(k) - x_j^{(l)}(k)=0},\,\,i,j = 1,2, \cdots ,N ,\,\,l = 1,2, \cdots ,n
\end{array}\]
holds for any $k\geq \rm{T}$.
\end{definition}

This paper aims to design control protocols and corresponding control gains to achieve the accelerated asymptotic consensus and the finite-time consensus.

\section{Accelerated asymptotic consensus by a time-invariant control protocol}

In this section, the accelerated asymptotic consensus problem of high-order MASs is studied.

Consider the following time-invariant control protocol
\begin{equation}
{u_i}(k) = K\sum\limits_{j \in {\mathcal{N}_i}} {{a_{ij}}({\bm{x}_j}(k) - {\bm{x}_i}(k))},
\end{equation}
where $K = [{K_1},{K _2}, \cdots ,{K _n}]\in \mathbb{R}^{1\times n}$ denotes the control gain, and ${\bm{x}_i}(k) = {[x_i^{(1)}(k), \cdots ,x_i^{(n)}(k)]^T} $.
Denote $\bm{x}(k)=[\bm{x}_{1}(k)^T,\bm{x}_{2}(k)^T,\cdots,\bm{x}_{N}(k)^T]^T$.
The system (2) can be written as
\begin{equation}
\bm{x}(k + 1) = ({I_N} \otimes A - \mathcal{L} \otimes BK)\bm{x}(k).
\end{equation}

Let $H({\lambda _i},K) = A \!-\! {\lambda _i}BK $.
According to ${\mathcal{L}=V\Lambda V^{T}}$, we have
\begin{equation}
\begin{aligned}
\bm{x}(k) =& (V \otimes {I_n})diag\{ A,H({\lambda _2},K), \ldots ,H({\lambda _N},K)\}
\\&({V^T} \otimes {I_n})\bm{x}(k - 1)
\\=& \frac{1}{N}(\bm{1}_N \otimes {I_n}){A^{k}}({\bm{1}^T_N} \otimes {I_n})\bm{x}(0)
\\&+\sum\limits_{i = 2}^N {({\bm{v_i}} \otimes {I_n}){H^k}({\lambda _i},K)(\bm{v_i}^T \otimes {I_n})x(0)} .
\end{aligned}
\end{equation}
Note that $\frac{1}{N}(\bm{1}_N \otimes {I_n}){A^{k}}({\bm{1}^T_N} \otimes {I_n})\bm{x}(0)$ in (6) is the part to achieve consensus.
We need to design the control gain $K$ so that
\[\mathop {\lim }\limits_{k \to \infty } \sum\limits_{i = 2}^N {(\bm{v_i} \otimes {I_n}){H^k}({\lambda _i},K)(\bm{v_i}^T \otimes {I_n})x(0)}  = 0.\]
The convergence rate is determined by the eigenvalue of $H({\lambda _i},K)$ with the largest modulus.
Thus, the consensus convergence rate can be defined as \cite{2011Network}
\begin{equation}
{r} = r(K)= \mathop {\max }\limits_{{\lambda _i} \in \left\{ {{\lambda _2}, \ldots ,{\lambda _N}} \right\}} \rho \left( {H({\lambda _i},K)} \right),
\end{equation}
where $\rho \left(  \cdot  \right)$ denotes spectral radius.

\begin{remark}
Denote \[\bm{e}(k)=\bm{x}(k)-\frac{1}{N}({\bm{1}_N}\otimes{I_n}){A^k}(\bm{1}_N^{T}\otimes{I_n})\bm{x}(0).\]
According to equation (6), we have $\mathop {\lim }\limits_{k \to \infty } {\left\| {\bm{e}(k)} \right\|_2} \sim {\rm{O}}({r^k})$.
Note that time $t=k\tau$ at step $k$.
Then $\mathop {\lim }\limits_{t \to \infty } {\left\| {\bm{e}(t)} \right\|_2} \sim {\rm{O}}({r^{t/\tau}})$.
A smaller $r$ or $\tau$ can get faster convergence of consensus.
\end{remark}

In this section, our goal is to design the control gain $K$ to make the convergence rate $r$ as small as possible.

\subsection{Conditions for consensus}

In this subsection, the necessary and sufficient condition for higher-order MASs to achieve asymptotic consensus is proposed.

\begin{lemma} (Schur Complement \cite{boyd2004convex}) Given a matrix \[M = \left[ {\begin{array}{*{20}{c}}M_1&M_2\\M_3&M_4\end{array}}\right],\] with nonsigular $M_1\in {\mathbb{R}^{\mu \times \mu}}$, $M_2 \in {\mathbb{R}^{N \times \mu}}$, $M_3 \in {\mathbb{R}^{\mu \times N}}$, and $M_4 \in {\mathbb{R}^{N \times N}}$. Then
$\det M = \det M_1 \cdot \det (M_4 - M_3{M_1^{ - 1}}M_2)$.
\end{lemma}

\begin{lemma}
Consider the high-order MAS (1) on a connected graph $\mathcal{G}$ with the control protocol (4). Let $0=\lambda_{1}<\lambda_{2}\leq\cdots\leq\lambda_{N}$ be the eigenvalues of the graph Laplacian matrix.
Then
\\(i) consensus can be achieved asymptotically if and only if $r < 1$;
\\(ii) the final consensus state is
\begin{equation}
\mathop {\lim }\limits_{k \to \infty } x(k) = \bm{1}_N \otimes \left[ {\begin{array}{*{20}{c}}
{\mathop {\lim }\limits_{k \to \infty } {s_1}(k)}, \cdots,
{\mathop {\lim }\limits_{k \to \infty } {s_n}(k)}
\end{array}} \right]^T,
\end{equation}
where \[{s_j}(k) = \frac{1}{N}\sum\limits_{m = 1}^{n - j + 1} {{\tau ^{m - 1}}C_k^{m - 1}\sum\limits_{p = 1}^N {x_p^{(m + j - 1)}(0)} }, j = 1, \ldots ,n.\]
\end{lemma}
\begin{proof}
(i)
Note that
$\mathop {\lim }\limits_{k \to \infty } {[H(\lambda_i,K)]^k} = \bm{0}_{n \times n}$, if and only if $\rho(H(\lambda_i,K))<1$.
Then $\mathop {\lim }\limits_{k \to \infty } {[H(\lambda_i,K)]^k} = \bm{0}_{n \times n}$ holds for all $i$, if and only if $r<1$.
It follows from (6) that
\begin{equation}
\begin{aligned}
\mathop {\lim }\limits_{{k} \to \infty } \bm{x}({k})
&= \frac{1}{N}({\bm{1}_N} \otimes {I_n})\mathop {\lim }\limits_{{k} \to \infty } {A^{k}}(\bm{1}_N^T \otimes {I_n})\bm{x}(0)
\\&= \frac{1}{N} \mathop {\lim }\limits_{{k} \to \infty } {{\bm{1}_N \bm{1}_N^{T}}} \otimes {A^{k}}\bm{x}(0),
\end{aligned}
\end{equation}
holds if and only if $r<1$.
Equation (9) is equivalent to
\begin{equation}
\mathop {\lim }\limits_{{k} \to \infty } {\bm{x}_i}({k}) = \frac{1}{N}  \mathop {\lim }\limits_{{k} \to \infty } {A^k}\sum\limits_{p = 1}^N {{\bm{x}_p}(0)}, \,i = 1,\ldots,N,
\end{equation}
which implies $\mathop {\lim }\limits_{k \to \infty } [ {x_i^{(l)}(k) - x_j^{(l)}(k)} ] = 0$.
Thus, consensus can be achieved asymptotically if and only if $r < 1$.
\\(ii) By direct computation, we have
\begin{equation}
{A^k} = \left[ {\begin{array}{*{20}{c}}
1&{C_k^1\tau }&{C_k^2{\tau ^2}}& \cdots &{C_k^{n - 1}{\tau ^{n - 1}}}\\
{}&1&{C_k^1\tau }& \ddots & \vdots \\
{}&{}& \ddots &{C_k^1\tau }&{C_k^2{\tau ^2}}\\
{}&{}&{}&1&{C_k^1\tau }\\
{}&{}&{}&{}&1
\end{array}} \right].
\end{equation}
Substituting (11) into (10), we get the consensus state as shown in (8).
\end{proof}
\begin{remark}
Note that the final state (8) is a kind of dynamic consensus.
The average consensus of the $l$-order state can be achieved, only if we set
$\sum\limits_{i = 1}^N {x_i^{(m)}(0)}  = 0,m = l + 1, \ldots ,n$.
\end{remark}

Next,
the direct link between the consensus convergence rate and the control gains is established as follows.

\begin{theorem}
Consider the high-order MAS (1) on a connected graph $\mathcal{G}$ with the control protocol (4).
Denote
\begin{equation}
\begin{aligned}
{R_i}(z,K) &=\det (zI - H({\lambda _i,K}))
\\&= {z^n} + b_1(\lambda_i){z^{n - 1}} +  \cdots b_{n - 1}(\lambda_i)z + b_n(\lambda_i) ,
\end{aligned}
\end{equation}
where
\begin{equation}
\begin{array}{l}
b_j(\lambda_i) = {\lambda _i}\sum\limits_{p = 1}^j {{{( - 1)}^{j - p}}{\tau ^p}{K_{n + 1 - p}}C_{n - p}^{n - j}}  + {( - 1)^j}C_n^{n - j},
\\i = 2, \ldots ,N,j = 1, \ldots ,n.
\end{array}
\end{equation}
Then consensus is achieved asymptotically if and only if
all the roots of
${R_i}(z,K)=0,i=2,\ldots,N$
are within the unit circle.
\end{theorem}

\begin{proof}
Applying the Schur Complement, we have
\[\det (zI - H({\lambda _i},K)) = \det P\cdot\det (U - Y{P^{ - 1}}Q),\]
where
\[\begin{array}{l}
P = \left[ {\begin{array}{*{20}{c}}
{z - 1}&{ - \tau }&{}&{}&{}\\
{}&{z - 1}&{ - \tau }&{}&{}\\
{}&{}& \ddots & \ddots &{}\\
{}&{}&{}&{z - 1}&{ - \tau }\\
{}&{}&{}&{}&{z - 1}
\end{array}} \right],\\
U = z - 1 + {\lambda _i}\tau {K_n},\\
Y = \left[ {{\lambda _i}\tau {K_1},{\lambda _i}\tau {K_2},...,{\lambda _i}\tau {K_{n - 1}}} \right],\\
Q = {\left[ {0, \ldots ,0, - \tau } \right]^T}.
\end{array}\]
After some determinant calculations, we get
\begin{equation}
\det (zI - H({\lambda _i}),K) = {(z - 1)^n} + {\lambda _i}\sum\limits_{p = 1}^n {{\tau ^p}{K _{n \!-\! p \!+\! 1}}{{(z - 1)}^{n - p}}},
\end{equation}
which can be expanded into (12).
The roots of ${R_i}(z,K)=0$ are within the unit circle if and only if $\rho( H({\lambda _i}),K)<1$.
Finally, it follows from Lemma 3 that consensus is achieved if and only if the roots of ${R_i}(z)=0,i=2,\ldots,N$ are within the unit circle.
\end{proof}

\subsection{Optimization of the convergence rate}

In this subsection, an accelerated consensus algorithm based on gradient descent is designed to optimize the convergence rate.
The lower bound of the convergence rate is given, and explicit control gains are derived as the necessary
condition to achieve the optimal convergence rate.

The goal of the accelerated consensus algorithm is to design the control gain $K$ so that the convergence rate $r(K)$ is as small as possible under the consensus condition, that is,
\begin{equation}
\begin{array}{*{20}{c}}
{\mathop {\min }\limits_{K} r(K)}\\
{\begin{array}{*{20}{c}}
{s.t.}&{\rho \left( {H({\lambda _i},K)} \right) <1 , i=2,\ldots,N.}
\end{array}}
\end{array}
\end{equation}

Denote
$
\nabla r = {[\nabla {r_1}, \ldots ,\nabla {r_n}]^T}, \nabla {r_m} = \frac{{r(K + \bm{\delta}(m)) - r(K)}}{{{\delta}}},
$
where $\bm{\delta}(m) \in \mathbb{R}^{1\times n}$ is a row vector whose elements are $0$ except the $m$-th term is a tiny positive scalar $\delta$.
Then the accelerated consensus algorithm described in Algorithm 1 can provide a numerical solution to the optimization problem (15).

\begin{algorithm}[h]
	\caption{Accelerated Consensus Algorithm Based on Gradient Descent}
	\label{alg::Gradient}
	\begin{algorithmic}[1]
		\Require
        nonzero eigenvalues of graph Laplacian $\lambda_i$;
         number of iterations $T$;
        sampling period $\tau$;
         a tiny positive scalar $\delta$;
         the learning rate $\eta$;
         initial parameters $K_i^{(0)}$
		\Ensure
		$r^*=r(K^{(T)})$, $K^*=K^{(T)}$
    \State Calculate $r(K^{(0)})$, $\nabla r^{(0)}_m=\frac{{r(K^{(0)} + \bm{\delta}(m)) - r(K^{(0)})}}{{{\delta}}}$
    \State \textbf{for} $t=1$ to $T$ do
    \State ~~~~$K^{(t)}= K^{(t-1)}- \eta \nabla r^{(t-1)}$
    \State ~~~~$\nabla r^{(t)}_m= \frac{{r(K^{(t)}) + \bm{\delta}(m)) - r(K^{(t)})}}{{{\delta}}}$ for all $m=1,\ldots,n$
    \State \textbf{end for}
	\end{algorithmic}
\end{algorithm}

The convergence rate in Algorithm 1 is bounded.
Next, we give a lower bound on the convergence rate, and the necessary condition for reaching this convergence rate.

\begin{theorem}
Consider the high-order MAS (1) on a connected graph ${\mathcal{G}}$ with the control protocol (4).
The following conclusions hold.
\\(i) The consensus convergence rate has a lower bound
\begin{equation}
{r} \ge {\left( {\frac{{{\lambda _N} - {\lambda _2}}}{{{\lambda _N} + {\lambda _2}}}} \right)^{1/n}}.
\end{equation}
\\(ii) The convergence rate ${r^*} = {\left( {\frac{{{\lambda _N} - {\lambda _2}}}{{{\lambda _N} + {\lambda _2}}}} \right)^{1/n}}$ can be achieved only if the control gains are

\begin{equation}
\begin{array}{l}
K_j^* =\frac{1}{{{\tau ^{n \!+\! 1 \!-\! j}}}}({f_{n + 1 - j}} + \sum\limits_{i = 1}^{n - j} {K_{j + i}^*{{( - 1)}^{i + 1}}{\tau ^{n - j + 1 - i}}C_{j - 1 + i}^{j - 1}} ),\\
K_n^* = \frac{{{1}}}{\tau }f_1,j =1, \ldots ,n - 1,
\end{array}
\end{equation}
where
\[\begin{array}{l}
{f_q} = \frac{{{{( - 1)}^q}}}{{2{\lambda _2}{\lambda _N}}}[{({r^*})^{ - n + 2q}}C_n^q({\lambda _N} - {\lambda _2}) - C_n^{n - q}({\lambda _N} + {\lambda _2})],\\
q = 1, \ldots ,n.
\end{array}\]
\end{theorem}

\begin{proof}
(i)
Let $z = r\frac{{s + 1}}{{s - 1}}$ in (12), and have
\begin{equation}
\tilde R_i(s,K) = c_0(\lambda_i){s^n} + c_1(\lambda_i){s^{n - 1}} +  \cdots  + c_{n - 1}(\lambda_i)s + c_n(\lambda_i),
\end{equation}
where
\[\begin{array}{l}
c_j(\lambda_i) = C_n^{n - j} + \sum\limits_{q = 1}^n {{w_{pq}}(r)b_q(\lambda_i)} ,\\{w_{pq}}(r) = {r^{n - q}}{{\tilde w}_{pq}},\\
{{\tilde w}_{pq}} = \sum\limits_{j = 0}^{p - 1} {C_{n - q}^{n - q - j}C_q^{q - p + 1 + j}{{\left( { - 1} \right)}^{p - 1 - j}}},
\end{array}\]
and $b_q(\lambda_i)$ is defined in (13).
The roots of $\tilde R_i(s,K)=0$ are in the left plane, if and only if the roots of $R_i(z,K)=0$ are in the circle with radius $r$.
Let $\bm{c}(\lambda_i)=[c_0(\lambda_i),c_1(\lambda_i),\ldots,c_n(\lambda_i)]^T$.
Substituting (13) into $c_j(\lambda_i)$, the vector form of the linear relationship between $c_j(\lambda_i)$ and $K_j$ can be written as
\begin{equation}
{\bm{c}(\lambda_i)} = \bm{h}(r) + {\lambda _i}{W(r)}{M}K^T,
\end{equation}
where
\[\begin{array}{l}
\bm{h}(r) = [ {h_1}(r), \ldots ,{h_{n + 1}}(r) ]^T,\\
{h_p}(r)\! = \![ {{r^n} \!+\! {{( - 1)}^{n \!+\! p \!-\! 1}}} ]C_n^{n - p + 1}\! + \!\sum\limits_{q = 1}^{n \!-\! 1} {{{w}_{pq}}(r){{( - 1)}^q}C_n^{n \!-\! q}},  \\
W(r) = [w_{pq}(r)] \in {\mathbb{R}^{(n + 1) \times n}},\\
M = {\left[ {\begin{array}{*{20}{c}}
{}&{}&{}&{}&\tau \\
{}&{}&{}&{{\tau ^2}}&{ - \tau C_{n - 1}^{n - 2}}\\
{}&{}&{{\mathinner{\mkern2mu\raise1pt\hbox{.}\mkern2mu
 \raise4pt\hbox{.}\mkern2mu\raise7pt\hbox{.}\mkern1mu}} }&{ - {\tau ^2}C_{n - 2}^{n - 3}}&{\tau C_{n - 1}^{n - 3}}\\
{}&{{\tau ^{n - 1}}}&{{\mathinner{\mkern2mu\raise1pt\hbox{.}\mkern2mu
 \raise4pt\hbox{.}\mkern2mu\raise7pt\hbox{.}\mkern1mu}}}& \vdots & \vdots \\
{{\tau ^n}}&{ - {\tau ^{n - 1}}}& \cdots &{{{\left( { - 1} \right)}^{n - 2}}{\tau ^2}}&{{{\left( { - 1} \right)}^{n - 1}}\tau }
\end{array}} \right]_{{n \times n}}}
\end{array}\]

Before the next step of the proof,
We make some notes on the coefficients of polynomial (18).
${\tilde w}_{pq}$ represents the coefficient of $s^{n-p+1}$ of the polynomial $(s+1)^{n-q}(s-1)^{q}$,
and ${( - 1)^{p - 1}}{{\tilde w}_{pq}}\!= \!{{\tilde w}_{p(n - q)}}$.
By setting $s\!=\!1$, we have $\sum\nolimits_{p = 1}^{n + 1} {{{\tilde w}_{pq}}}  \!=\! 0$, that is, the column sum of the matrix
$W(r)\! =\! [ {{w_{pq}(r)}} ] \!\in\! {\mathbb{R}^{(n \!+\! 1) \times n}}$ is zero.
Similarly, by setting $s=1$ and $s=-1$, we can get that for the $q$-th ($q \ne n$) column of matrix $W(r)$, the sum of odd elements is zero, and the sum of even elements is zero.
These two properties will be used to obtain (23) and (24), respectively.

Applying the Routh-Hurwitz stability criterion \cite{1100537}, the polynomial (18) is stable or marginally stable only if
\begin{equation}
\bm{c}(\lambda_i) \ge \bm{0}, i = 2, \ldots ,N.
\end{equation}
Note that (20) are linear inequalities about $\lambda_i$.
We only need to require
\begin{equation}
\begin{array}{l}
\bm{c}(\lambda_i) \ge \bm{0}, i = 2,N.
\end{array}
\end{equation}
to satisfy (20).
Assume that $n$ is odd.
According to the inequality properties, the following inequality (22) holds.
\begin{equation}
\begin{aligned}
&\frac{{c_0(\lambda_2)\! +\! c_2(\lambda_2)\! +\!  \cdots  \!+\! c_{n \!-\! 1}(\lambda_2)}}{{{\lambda _2}}}
\\&+ \frac{{c_1(\lambda_N) \!+\! c_3(\lambda_N) \!+\!  \cdots  \!+\! c_n(\lambda_N)}}{{{\lambda _N}}}\! \ge\! 0.
\end{aligned}
\end{equation}
Since the column sum of $W(r)$ is zero, the sum of the vector ${W(r)}{M}K^T$ is zero.
According to (19), the left side of the inequality (22) can be written as
\begin{equation}
\frac{{{h_1}(r)\! +\! {h_3}(r) \!+\!  \cdots\!  +\! {h_n}(r)}}{{{\lambda _2}}}
+ \frac{{{h_2}(r) \!+\! {h_4}(r) \!+\!  \cdots  \!+\! {h_{n + 1}}(r)}}{{{\lambda _N}}}
\end{equation}
By some algebraic calculations, we get
\begin{equation}
\begin{array}{l}
{h_1}(r) + {h_3}(r) +  \cdots  + {h_n}(r) = {2^{n - 1}}({r^n} - 1),\\
{h_2}(r) + {h_4}(r) +  \cdots  + {h_{n + 1}}(r) = {2^{n - 1}}({r^n} + 1).
\end{array}
\end{equation}
Then the inequality (22) can be written as
\begin{equation}
{2^{n - 1}}\left( {\frac{{{r^n} - 1}}{{{\lambda _2}}} + \frac{{{r^n} + 1}}{{{\lambda _N}}}} \right) \ge 0.
\end{equation}
It follows from (25) that  ${r} \ge {\left( {\frac{{{\lambda _N} - {\lambda _2}}}{{{\lambda _N} + {\lambda _2}}}} \right)^{1/n}}$.
Similarly, for the case where $n$ is even, we can obtain the same lower bound.
The proof of part (i) is completed.
\\(ii)
For the case where $n$ is odd,
$r=r^*$ holds only if
\begin{equation}
\begin{array}{l}
c_0(\lambda_2) = 0,c_2(\lambda_2) = 0, \ldots ,c_{n \!-\! 1}(\lambda_2) = 0,\\
c_1(\lambda_N) = 0,c_3(\lambda_N) = 0, \ldots ,c_n(\lambda_N) = 0.
\end{array}
\end{equation}
Similarly, for the case where $n$ is even, $r=r^*$ holds only if
\begin{equation}
\begin{array}{*{20}{l}}
{c_1(\lambda_2) = 0,c_3(\lambda_2) = 0, \ldots ,c_{n \!-\! 1}(\lambda_2) = 0,}\\
{c_0(\lambda_N) = 0,c_2(\lambda_N) = 0, \ldots ,c_n(\lambda_N) = 0.}
\end{array}
\end{equation}
To combine (26) and (27), we denote
\[J = diag\left\{ {{J_1},{J_2}, \ldots ,{J_{n + 1}}} \right\}\in {\mathbb{R}^{(n + 1) \times (n + 1)}},\]
where
\[{J_i} = \frac{{2{\lambda _2}{\lambda _N}}}{{({\lambda _N} + {\lambda _2}) + {{( - 1)}^{n - i}}({\lambda _N} - {\lambda _2})}},i = 1, \ldots ,n + 1.\]
Then $r=r^*$ holds for any $n$ only if euqation (28) holds.
\begin{equation}
h(r^*)+JW(r^*)M K^T=0.
\end{equation}

Next, we solve $K$ in equation (28).
Since $J$ is invertible, (28) is equivalent to
\begin{equation}
W({r^*})M{K^T} = -{J^{ - 1}}\bm{h}({r^*}).
\end{equation}
The $p$-th row of $-{J^{ - 1}}\bm{h}({r^*})$ can be written as

\begin{small}
\begin{equation}
\begin{aligned}
 &  - \!\frac{1}{{{J_p}}}\sum\limits_{q \! =\! 1}^{n \!-\! 1} {{{\tilde w}_{p(n\! -\! q)}}{{({r^*})}^{n \!-\! q}}{{( \!-\! 1)}^{1 \!+\! q \!- \!p}}C_n^{n \!-\! q}}  + C_n^{n \!+\! 1 \!-\! p}\frac{{{{( -\! 1)}^{n \!-\! p}} \!\times \!2}}{{{\lambda _N} \!+\! {\lambda _2}}}
\\=& \frac{1}{{2{\lambda _2}{\lambda _N}}}\sum\limits_{q = 1}^{n \!-\! 1} {{{({r^*})}^q}C_n^q{{\tilde w}_{pq}}{{( - \!1)}^q}[({\lambda _N} \!-\! {\lambda _2}) \!+\! {{( - \!1)}^{n \!+\! p}}({\lambda _N} \!+\! {\lambda _2})]}
\\& + C_n^{n \!+\! 1\! -\! p}\frac{{{{( -\! 1)}^{n \!-\! p}} \times 2}}{{{\lambda _N} \!+\! {\lambda _2}}}
\\=& \sum\limits_{q = 1}^{n \!-\! 1} {\frac{{{{\tilde w}_{pq}( - \!1)^q}}}{{2{\lambda _2}{\lambda _N}}}[{{({r^*})}^q}C_n^q({\lambda _N} \!-\! {\lambda _2})\! - \! {{({r^*})}^{n \!-\! q}}C_n^{n \!-\! q}({\lambda _N} \!+\! {\lambda _2})]}
\\&+ \!\frac{{{{( - \!1)}^{p \!+ \!n \!-\! 1}}C_n^{n \!+\! 1 \!-\! p}}}{{2{\lambda _2}{\lambda _N}}}\left[ {{{\left( {{r^*}} \right)}^n}({\lambda _N} \!-\! {\lambda _2})\! -\! ({\lambda _N} \!+\! {\lambda _2})} \right]
\\= &\sum\limits_{q = 1}^n { \frac{{{w_{pq}(r^*)( \!-\! 1)^q}}}{{2{\lambda _2}{\lambda _N}}}[{{({r^*})}^{ - \!n \!+ \!2q}}C_n^q({\lambda _N} \!-\! {\lambda _2})\! -\! C_n^{n \!-\! q}({\lambda _N} \!+\! {\lambda _2})]}.
\end{aligned}
\end{equation}
\end{small}
It follows from (30) that
\begin{equation}
\sum\limits_{q = 1}^n {{w_{pq}(r^*)} f_q}=- \frac{1}{{{J_p}}}{h_p}({r^*}),\,\,p=1,\ldots,n+1.
\end{equation}
Then we have
\begin{equation}
\bm{f}=M{\bm{K}^T},
\end{equation}
where $\bm{f} ={\left[ {{f_1}, \ldots ,{f_n}} \right]^T}$.
Equation (32) can be expanded to
\begin{equation}
\begin{aligned}
  {f_1}=&{K_n}\tau , \\
  {f_2}=&{K_{n - 1}}{\tau ^2} - {K_n}\tau C_{n - 1}^{n - 2}, \\
   \cdots  \\
  {f_i}=&{K_{n - i + 1}}{\tau ^i} - {K_{n - i + 2}}{\tau ^{i - 1}}C_{n - i + 1}^{n - i} \hfill
  +  \cdots  \\
        &  + {K_n}{( - 1)^{i - 1}}\tau C_{n - 1}^{n - i}, \hfill
   \\ \cdots  \\
{f_n}= &{K_1}{\tau ^n} - {K_2}{\tau ^{n - 1}} +  \cdots  + {K_n}{{( - 1)}^{n - 1}}\tau .
\end{aligned}
\end{equation}
It follows from (33) that the solution of (28) is (17).
The proof of part (ii) is completed.
\end{proof}

In Theorem 2, we give the necessary condition to achieve the lower bound on the convergence rate.
When $n=1$, this condition is sufficient and necessary \cite{XIAO200465}.
When $n=2$, we can prove that the condition (17) is also sufficient and necessary, as shown in Corollary 1.
Moreover, if the high-order MAS is on a star graph, then the the condition (17) is sufficient and necessary for any $n$, as shown in Corollary 2.

\begin{corollary}
Consider the MAS (1) on a connected graph ${\mathcal{G}}$ with the control protocol (4).
The optimal convergence rate of $n=2$ is
\begin{equation}
{r^*} = \sqrt {{\frac{{{\lambda _N} - {\lambda _2}}}{{{\lambda _N} + {\lambda _2}}}}}
\end{equation}
with the following control gains
\begin{equation}
{K_1^*} = \frac{{2\lambda _2 }}{{{\tau^2}(\lambda _2  + \lambda _N )\lambda _N }},
{K_2^*} = \frac{2}{{\lambda _N \tau}}.
\end{equation}
\end{corollary}

\begin{proof}
When $n=2$, the polynomial (18) becomes
\begin{equation}
\begin{aligned}
\tilde R_i(s,K) = &[{r^2} + ({\lambda _i}\tau {K_2} - 2)r + 1 + {\lambda _i}{\tau ^2}{K_1} - {\lambda _i}\tau {K_2}]{s^2}
\\&+ [2{r^2} - 2(1 + {\lambda _i}{\tau ^2}{K_1} - {\lambda _i}\tau {K_2})]{s}
\\&+ {r^2} - ({\lambda _i}\tau {K _2} - 2)r + 1 + {\lambda _i}{\tau ^2}{K_1} - {\lambda _i}\tau {K_2}.
\end{aligned}
\end{equation}
According to the Routh-Hurwitz stability criterion, (36) is stable or marginally stable if and only if
\begin{equation}
\begin{aligned}
&{r^2} + ({\lambda _i}\tau {K_2} - 2)r + 1 + {\lambda _i}{\tau ^2}{K_1} - {\lambda _i}\tau {K_2}\ge 0,
\\&2{r^2} - 2(1 + {\lambda _i}{\tau ^2}{K_1} - {\lambda _i}\tau {K_2})\ge 0,
\\&{r^2} - ({\lambda _i}\tau {K _2} - 2)r + 1 + {\lambda _i}{\tau ^2}{K_1} - {\lambda _i}\tau {K_2}\ge 0.
\end{aligned}
\end{equation}
Since all constraints in (37) are all linear with respect to $\lambda_i$,
the inequalities can be reduced to
\begin{equation}
\begin{aligned}
&\tau {K_1} \!-\! (1 \!-\! r){K_2} \!+\! \frac{{{{(1 \!-\! r)}^2}}}{{{\lambda _N}\tau }} \!\ge\! 0,\\
& - \!2\tau {K_1} \!+\! 2{K_2} \!-\! \frac{{2 \!-\! 2{r^2}}}{{{\lambda _2}\tau }}     \! \ge\! 0,\\
&\tau {K_1}\! -\! (1 \!+\! r){K_2} \!+\! \frac{{{{(1\! +\! r)}^2}}}{{{\lambda _N}\tau }} \!\ge\! 0.
\end{aligned}
\end{equation}
Adding the three inequalities in (38), we have
\begin{equation}
\frac{{1 + {r^2}}}{{\lambda_N \tau }} - \frac{{1 - {r^2}}}{{\lambda_2 \tau }} \ge 0.
\end{equation}
It follows from (39) that $r \ge \sqrt {\frac{{\lambda_N  - \lambda_2}}{{\lambda_N  + \lambda_2 }}}$.
The optimal convergence rate  $r^* =  \sqrt {\frac{{\lambda_N  - \lambda_2}}{{\lambda_N  + \lambda_2 }}}$
is achieved if and only if
\begin{equation}
\begin{aligned}
&\tau {K_1} \!-\! (1 \!-\! r){K_2} \!+\! \frac{{{{(1 \!-\! r)}^2}}}{{{\lambda _N}\tau }} \!=\! 0,\\
 &- \!2\tau {K_1} \!+\! 2{K_2} \!-\! \frac{{2 \!-\! 2{r^2}}}{{{\lambda _2}\tau }}     \! =\! 0,\\
&\tau {K_1}\! -\! (1 \!+\! r){K_2} \!+\! \frac{{{{(1\! +\! r)}^2}}}{{{\lambda _N}\tau }} \!=\! 0.
\end{aligned}
\end{equation}
The solution to (40) is given by (35).
The proof is completed.
\end{proof}

\begin{remark}
In Corollary 1, we apply the Routh-Hurwitz stability criterion to derive the optimal convergence rate ${r^*} = \sqrt {{\frac{{{\lambda _N} - {\lambda _2}}}{{{\lambda _N} + {\lambda _2}}}}}$.
In \cite{2011Network,2014Closed}, authors obtained the same convergence rate, by analyzing the eigenvalues of the closed-loop matrix in the complex plane.
\end{remark}

\begin{corollary}
Consider the MAS (1) on a star graph ${\mathcal{G}}$ with the control protocol (4).
The optimal convergence rate is
$
{r^*} = {\left( {\frac{{{\lambda _N} - {\lambda _2}}}{{{\lambda _N} + {\lambda _2}}}} \right)^{1/n}}
$
with the optimal control gains (17).
\end{corollary}

\begin{proof}
For a star graph with $N$ nodes,
the Laplacian matrix has only two distinct non-zero eigenvalues $\lambda_2 = 1, \lambda_N = N$.
For the case where $n$ is odd,
when (26) holds,
the Hurwitz matrices of $\tilde R_2(s,K)$ and $\tilde R_N(s,K)$ are
\[\begin{array}{l}
{\Delta _1}({\lambda _2}) = {c_1}({\lambda _2}) \ge 0,{\Delta _2}({\lambda _2}) = \left| {\begin{array}{*{20}{c}}
{{c_1}({\lambda _2})}&{{c_3}({\lambda _2})}\\
{{c_0}({\lambda _2})}&{{c_2}({\lambda _2})}
\end{array}} \right| = 0,\\
 \cdots ,{\Delta _{n - 1}}({\lambda _2}) = 0,{\Delta _n}({\lambda _2}) = 0,
\end{array}\]
and
\[\begin{array}{l}
{\Delta _1}({\lambda _N}) = {c_1}({\lambda _N}) = 0,{\Delta _2}({\lambda _N})\! =\! \left| {\begin{array}{*{20}{c}}
{{c_1}({\lambda _N})}&\!{{c_3}({\lambda _N})}\\
{{c_0}({\lambda _N})}&\!{{c_2}({\lambda _N})}
\end{array}} \right| \!=\! 0,\\
 \cdots {\Delta _{n - 1}}({\lambda _N}) = 0,{\Delta _n}({\lambda _N}) = 0,
\end{array}\]
respectively.
Then the roots of $R_i(z,K)\!=\!0,i\!=\!2,N$ are on the circle of radius $r^*$ if and only if (26) holds.
Similarly, for the case where $n$ is even, the roots of $R_i(z,K)\!=\!0,i\!=\!2,N$ are on the circle of radius $r^*$ if and only if (27) holds.
Thus, the optimal convergence rate $r=r^*$ is achieved if and only if (28) holds, and the optimal control gains are given by (17).
\end{proof}

\section{Finite-time consensus by a time-varying control protocol}

In this section, a protocol with time-varying control is presented for high-order MASs to achieve consensus in  finite time.

Note that the finite-time consensus will be reached if and only if $R_i(z,K)=z^n=0$ holds for all $i=2,\ldots,N$ in (12).
However, the control protocol (4) with constant gains can't achieve it.
Therefore, we consider the following time-varying control protocol
\begin{equation}
{u_i}(k) = K(k)\sum\limits_{j \in {N_i}} {{a_{ij}}({x_j}(k) - {x_i}(k))},
\end{equation}
where $K(k)=[{K _1}(k),{K _2}(k), \cdots ,{K _n}(k)]\in \mathbb{R}^{1\times n}$.

\begin{lemma}
(Cayley-Hamilton \cite{mertzios1986generalized})
For a given $n\times n$ matrix  $H$, let $p(z)=\det (z{I_n} - H)$ be the characteristic polynomial of $H$. Then
$p(H) = \bm{0}_{n \times n}$.
\end{lemma}
\begin{theorem}
Consider the $n$-order MAS (1) on a connected graph $\mathcal{G}$ with the control protocol (41).
Assume that its Laplacian matrix has $\bar l$ distinct nonzero eigenvalues $\lambda_{p_l},l=0,\cdots,{\bar l}-1$, and
${\lambda _{{p_0}}} > {\lambda _{{p_1}}} >  \ldots  > {\lambda _{{p_{{\bar l} - 1}}}}$.
If the control gains are set as
\begin{equation}
\begin{array}{l}
{K_m}(nl + j) = \frac{{C_n^{m - 1}}}{{{\lambda _{{p_l}}}{\tau ^{n \!-\! m \!+\! 1}}}},
l = 0,\ldots ,\bar l \!-\! 1,\,\, j = 0, \ldots ,n \!-\! 1,\\
{K_m}(k) = 0, \,\, k \ge n \bar l
\end{array}
\end{equation}
for all $m = 1, \ldots n$, then consensus will be achieved at step $n\bar l$.
The consensus state is
\begin{equation}
x(k) = {\bm{1}_N} \otimes {\left[ {{s_1}(k), \ldots {s_n}(k)} \right]^T} , \,\, k\ge n \bar l
\end{equation}
 where
\[{s_j}(k) = \frac{1}{N}\sum\limits_{m = 1}^{n - j + 1} {{\tau ^{m - 1}}C_{k}^{m - 1}\sum\limits_{p = 1}^N {x_p^{(m + j - 1)}(0)} } ,j = 1, \ldots ,n.\]
\end{theorem}

\begin{proof}
The state of the agents has an iterative form
\begin{equation} \nonumber
\begin{aligned}
\bm{x}(t) = &(V \otimes {I_n})diag\{ {A^{t - 1}},{I_{{n_0}}} \otimes \prod\limits_{k = 0}^{t - 1} {[A - {\lambda _{{p_0}}}BK(k)]} ,
\\ & \ldots ,{I_{{n_{{\bar l} - 1}}}} \otimes \prod\limits_{k = 0}^{t - 1} {[A - {\lambda _{{p_{{\bar l} - 1}}}}BK(k)]} \} ({V^T} \otimes {I_n})\bm{x}(0).
\end{aligned}
\end{equation}
where $n_l$ denotes the algebraic multiplicity of $\lambda_{p_l}$, and $\sum\limits_{l = 0}^{{\bar l} - 1} {{n_l}}  = N - 1$.
Then consensus is achieved in finite time if and only if equation (44) holds.

\begin{equation}
\prod\limits_{k = 0}^{t-1} {\left[ {A{-}{\lambda _{{p_l}}}BK(k)} \right]}  = \bm{0}_{n \times n},\,\, l=0,1,\ldots,{\bar l}\!-\!1.
\end{equation}

Next, we design the control gain $K(nl),l=0,1,\ldots,\bar l - 1$ to satisfy
\[\det \left[ {zI- (A - {\lambda _{{p_l}}}BK(nl))} \right] = {z^n} = 0\]
and get
\begin{equation}
\begin{aligned}
&\prod\limits_{k = 0}^{n - 1} A - {\lambda _{{p_0}}}BK(k) = \prod\limits_{k = n}^{2n - 1} {A - {\lambda _{{p_1}}}BK(k)}  =  \cdots
\\&= \prod\limits_{k = n(\bar l - 1)}^{n\bar l - 1} {A - {\lambda _{{p_{\bar l - 1}}}}BK(k) = \bm{0}_{n \times n}} .
\end{aligned}
\end{equation}
Note that the characteristic equation of $A - {\lambda _{{p_l}}}BK(nl)$ can be written as
\begin{equation}
\begin{aligned}
&{(z - 1)^n} + {\lambda _{{p_l}}}\tau {K_n}(nl){(z - 1)^{n - 1}} +  \cdots
\\&+ {\lambda _{{p_l}}}{\tau ^{n - 1}}{K_2}(nl)(z - 1) + {\lambda _{{p_l}}}{\tau ^n}{K_1}(nl) = 0.
\end{aligned}
\end{equation}
Substituting $z=0$ into (46), we have
\begin{equation} \nonumber
\begin{aligned}
&{( - 1)^n} + {\lambda _{{p_l}}}\tau {K_n}(nl){( - 1)^{n - 1}} +  \cdots
\\& + {\lambda _{{p_l}}}{\tau ^{n - 1}}{K_2}(nl)( - 1) + {\lambda _{{p_l}}}{\tau ^n}{K_1}(nl) = 0.
\end{aligned}
\end{equation}
Since
${( - 1)^n} + C_n^{n - 1}{( - 1)^{n - 1}} +  \cdots  + C_n^1( - 1) + 1 =0$,
the characteristic equation (46) holds if we set
\begin{equation}
{K _m}(nl) = \frac{{C_n^{m - 1}}}{{{\lambda _{p_l}}{\tau ^{n - m + 1}}}},m = 1, \ldots ,n.
\end{equation}
According to the Cayley-Hamilton Theorem, we have
$[A - {\lambda_{p_l}}BK(nl)]^n ={\bm{0}_{n \times n}}$
when applying the control gain (47).
Thus, we design $K(nl+j)\equiv K(nl)$ at $j=0,1,\ldots,n-1$, and take the control gains (47) to satisfy (45).
Then the state at step $k\ge n\bar l$ is
\begin{equation}
\begin{aligned}
\bm{x}(k)&=\frac{1}{N}({\bm{1}_N} \otimes {I_n}){A^{k}}(\bm{1}_N^T \otimes {I_n})\bm{x}(0)
\\&= \frac{1}{N}({\bm{1}_N \bm{1}_N^T}) \otimes {A^{k}}\bm{x}(0).
\end{aligned}
\end{equation}
By substituting (11)
into (48), the consensus state (43) is obtained.
\end{proof}

\begin{remark}
The time-varying control sequence designed in (42) is similar to graph filtering \cite{2019Average}.
Large eigenvalues correspond to "high frequencies", and small eigenvalues correspond to "low frequencies".
In order to avoid too much oscillation during the convergence, we suggest to filter out the part with large eigenvalues first,
although the selection of the filtering order does not affect the final result in principle.
\end{remark}

\begin{remark}
Since the high-order MAS (1) can achieve consensus at step $n\bar l$ by applying control gain (42), the sampling period $\tau$ determines the overall convergence time.
Consensus will be achieved with arbitrarily fast convergence speed if $\tau \to 0$, which implies the infinite band-width communication.
In \cite{2018Accelerating}, authors have studied the accelerated consensus of high-order continuous-time systems and obtained the similar conclusion by allocating all the eigenvalues of the closed-loop system matrix.
\end{remark}

\begin{remark}
Reference \cite{2019Average} has proposed that the first-order MASs can achieve consensus at step $\bar l$ by applying the control gain
$K (k) = \frac{1}{{{\lambda _{{p_k}}}}}$.
This is a special case of $n = 1$ in (42).
Specially, if we consider a second-order MAS, it will reach the consensus state
\[\bm{x}(k) \!=\! {\bm{1}_N} \otimes {[\frac{1}{N}\sum\limits_{i = 1}^N {x_i^{(1)}(0)} \! +\! \frac{{k\tau }}{N}\sum\limits_{i = 1}^N {x_i^{(2)}(0)} ,\frac{1}{N}\sum\limits_{i = 1}^N {x_i^{(2)}(0)} ]^T}\]
$ k\ge 2\bar l$, by applying the control gain sequence
\[\begin{array}{l}
{K_1} \!=\! \left\{ {\frac{1}{{{\lambda _{{p_0}}}\!{\tau ^2}}},\frac{1}{{{\lambda _{{p_0}}}\!{\tau ^2}}},\frac{1}{{{\lambda _{{p_1}}}\!{\tau ^2}}},\frac{1}{{{\lambda _{{p_1}}}\!{\tau ^2}}}, \ldots ,\frac{1}{{{\lambda _{{p_{\bar l\! -\! 1}}}}\!{\tau ^2}}},\frac{1}{{{\lambda _{{p_{\bar l\! - \!1}}}}\!{\tau ^2}}}} \right\},\\
{K_2} \!= \!\left\{ {\frac{2}{{{\lambda _{{p_0}}}\!\tau }},\frac{2}{{{\lambda _{{p_0}}}\!\tau }},\frac{2}{{{\lambda _{{p_1}}}\!\tau }},\frac{2}{{{\lambda _{{p_1}}}\!\tau }}, \ldots ,\frac{2}{{{\lambda _{{p_{\bar l \!-\! 1}}}}\!\tau }},\frac{2}{{{\lambda _{{p_{\bar l \!-\! 1}}}}\!\tau }}} \right\}.
\end{array}\]
\end{remark}

\section{Numerical Examples}

This section uses two examples to verify the effectiveness of the proposed theory.

\textbf{Example 1:\,\,}
In this example,
Algorithm 1 is used to optimize the convergence rate on different graphs.
Consider a third-order MAS with $10$ agents on four unweighted graphs:
the graph $\mathcal{G}_1$ randomly generated by a small-world network model shown in Fig. 1,
the cycle graph $\mathcal{G}_2$,
the path graph $\mathcal{G}_3$,
the complete bipartite graph $\mathcal{G}_4$ with $4+5$ vertices.
Set $\tau=0.1$, $T=5000$, $\eta=0.01$, and $\delta  = 10^{-6}$.
Randomly initialize the control gain $K^{(0)}$ multiple times, run Algorithm 1, and record the optimal convergence rate.
The convergence rate $r^*$ obtained by Algorithm 1 is listed in Table I, where $r_{lb}$ represents the lower bound
${\left( {\frac{{{\lambda _N} - {\lambda _2}}}{{{\lambda _N} + {\lambda _2}}}} \right)^{1/3}}$.
It can be found that on different graphs, the convergence rate $r^*$ of the third-order MAS can reach the lower bound $r_{lb}$, and the smaller $\lambda_N/\lambda_2$ is, the smaller $r^*$ is.
Next, generate the initial state of each agent in the interval $[-1,1]$.
The convergence of the consensus error ${\left\| {\bm{e}(k)} \right\|_2}$
on different graphs is shown in Fig. 2.
It can be observed that the better the connectivity of the network is, the faster the consensus error converges.

\begin{figure}[!htbp]
\centering
\includegraphics[height=7cm]{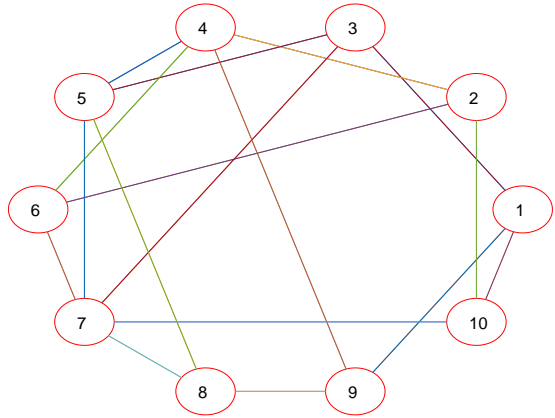}
\caption{The graph $\mathcal{G}_1$}
\label{fig:label}
\end{figure}

\begin{figure}[!htbp]
\centering
\includegraphics[height=7cm]{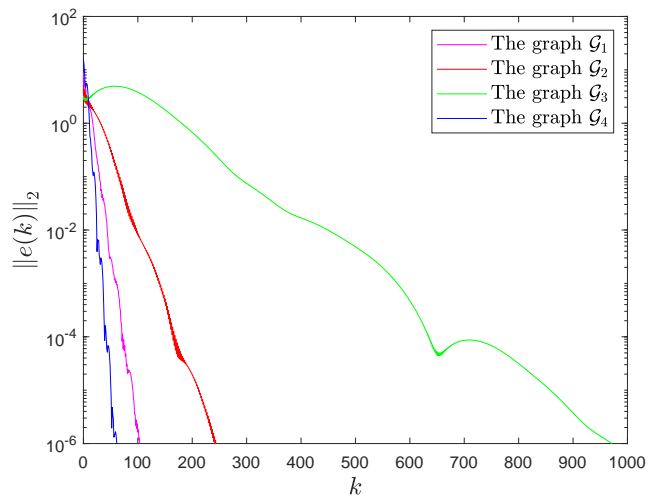}
\caption{Consensus error on different graphs}
\label{fig:label}
\end{figure}

\begin{table}[htbp]
\centering
  \caption{Convergence rate on different graphs}
\begin{tabular}{cccccc}
\toprule
                      &Graph $\mathcal{G}_1$   & Graph $\mathcal{G}_2$  & Graph $\mathcal{G}_3$  & Graph $\mathcal{G}_4$ \\ \midrule
$\lambda_N/\lambda_2$ & 4.4790                       & 10.4721          & 39.8635          &  2.5000         \\
$r_{lb}$              & 0.8595                        & 0.9381           & 0.9834           &  0.7539         \\
$r^*$                 & 0.8595                        & 0.9381           & 0.9834           &  0.7539        \\
\bottomrule
\end{tabular}
\label{tab:addlabel}
\end{table}

\textbf{Example 2:\,\,}
In this example, the effectiveness of the control strategy in Theorem 3 is verified.
Consider the cycle graph $\mathcal{G}_2$ with $10$ nodes.
The distinct non-zero eigenvalues of Laplacian matrix $\mathcal{L}_{\mathcal{G}_2}$ are $\{0.3820,1.3820,2.6180,3.6180,4\}$.
Let $\tau=0.1$. Randomly set the initial state of the agents in the interval $[-5,5]$.
According to Theorem 3, when $n=2$, consensus will be achieved at step $n\bar l = 10$ as shown in Fig. 3, and the final consensus state of agent $i$ is
$\bm{x}_i(k)= [1.3325+0.0963k,0.9627]^T,k\ge 10.$
Similarly, when $n=3$, consensus will be achieved at step $n\bar l = 15$ as shown in Fig. 4, and the consensus state of agent $i$ is
$\bm{x}_i(k)= [1.3325+0.0850k+0.0113k^2,0.9627+0.2266k,2.2662]^T,k\ge 15.$

	\begin{figure*}[!t]
		\vspace{4pt}
		\subfigure[\scriptsize Position]{
			\begin{minipage}[t]{0.49\linewidth}
				\label{fig1:1}
				\centering
				\centerline{\includegraphics[width=8cm]{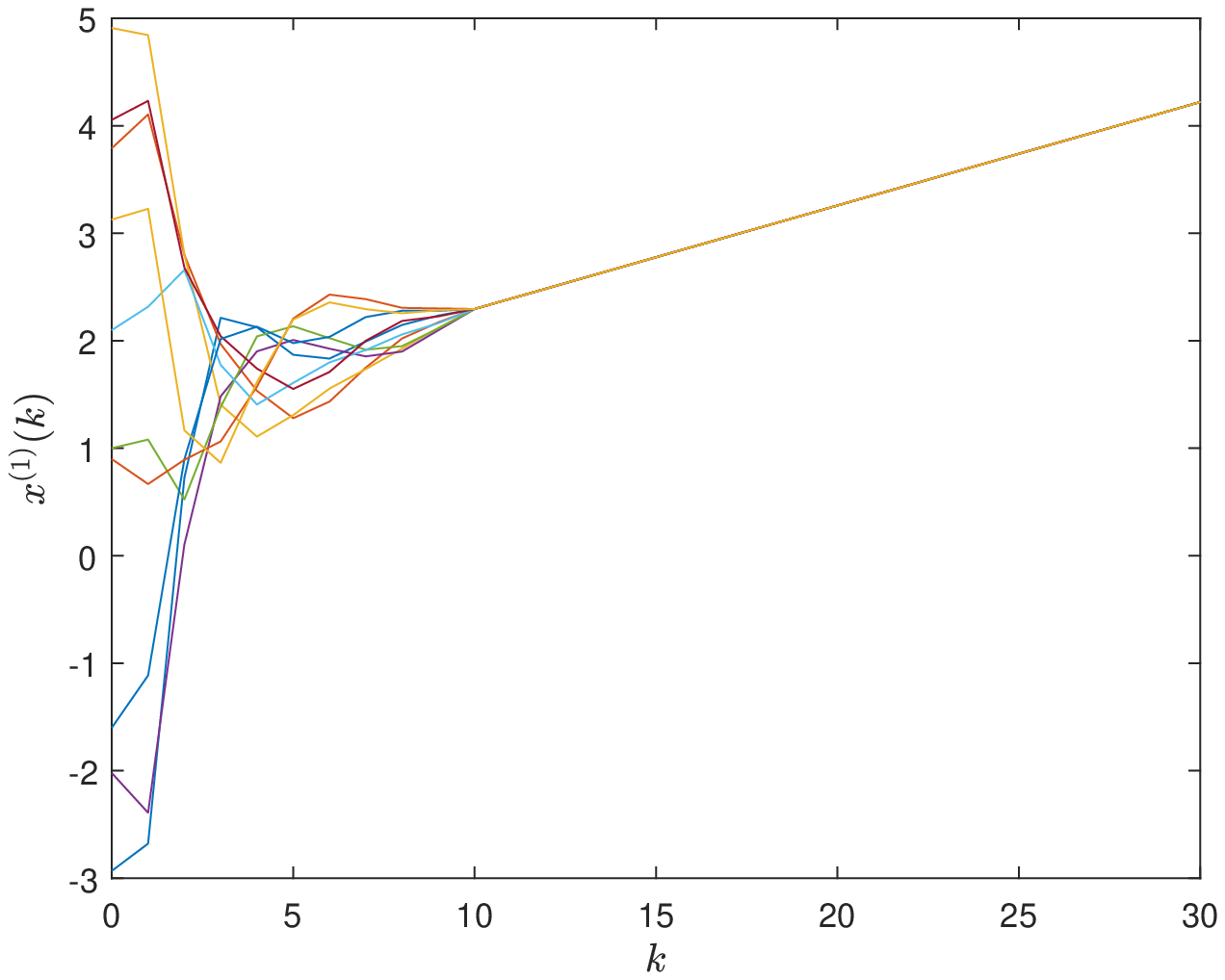}}
			\end{minipage}%
		}
		\subfigure[\scriptsize Velocity]{
			\begin{minipage}[t]{0.49\linewidth}
				\label{fig1:2}
				\centering
				\centerline{\includegraphics[width=8cm]{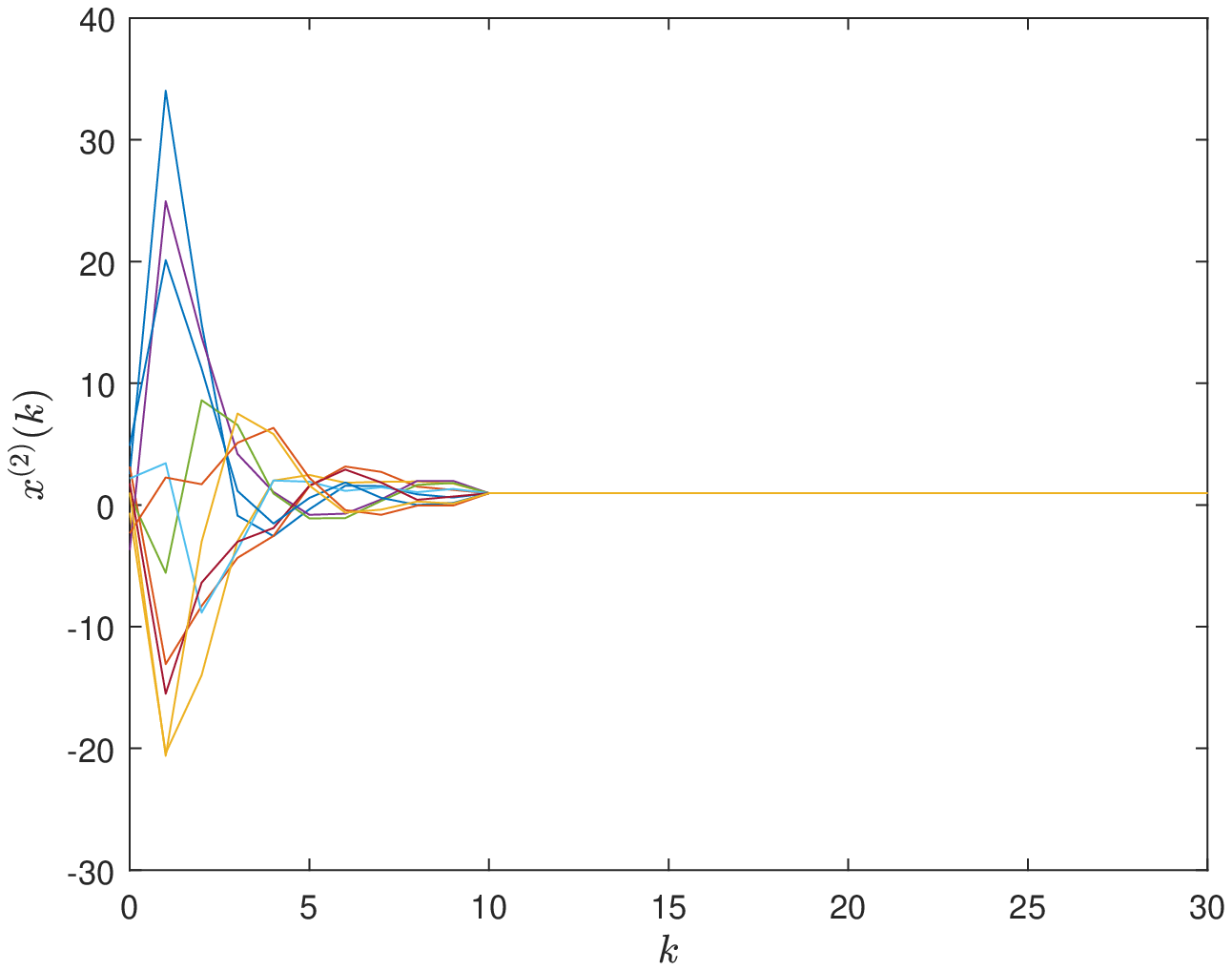}}
			\end{minipage}
		}
		\caption{Finite-time consensus of the second-order MAS}
		\label{}
	\end{figure*}

	\begin{figure*}[!t]
		\vspace{4pt}
		\subfigure[\scriptsize Position]{
			\begin{minipage}[t]{0.33\linewidth}
				\label{fig1:1}
				\centering
				\centerline{\includegraphics[width=5.4cm]{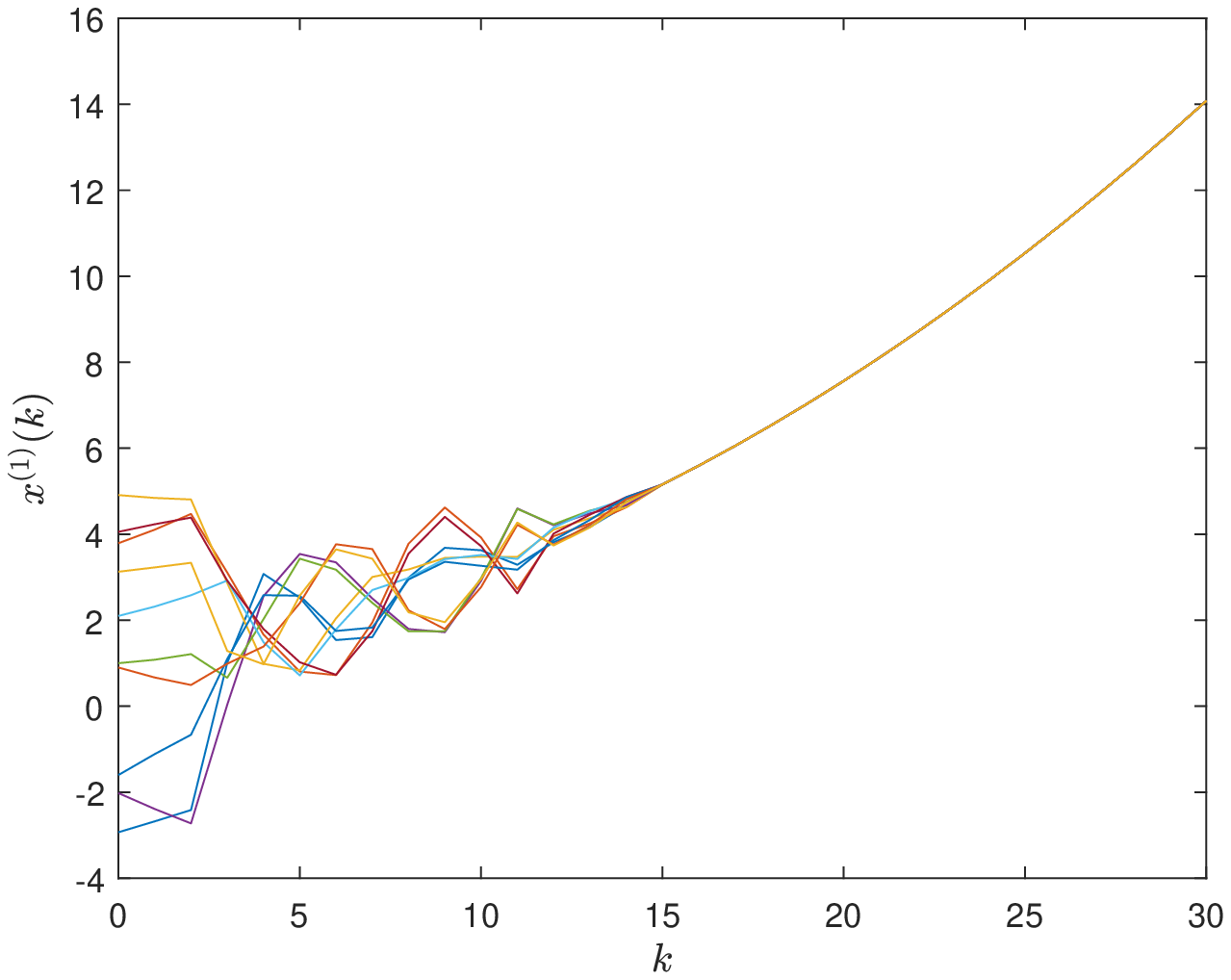}}
			\end{minipage}%
		}
		\subfigure[\scriptsize Velocity]{
			\begin{minipage}[t]{0.33\linewidth}
				\label{fig1:2}
				\centering
				\centerline{\includegraphics[width=5.4cm]{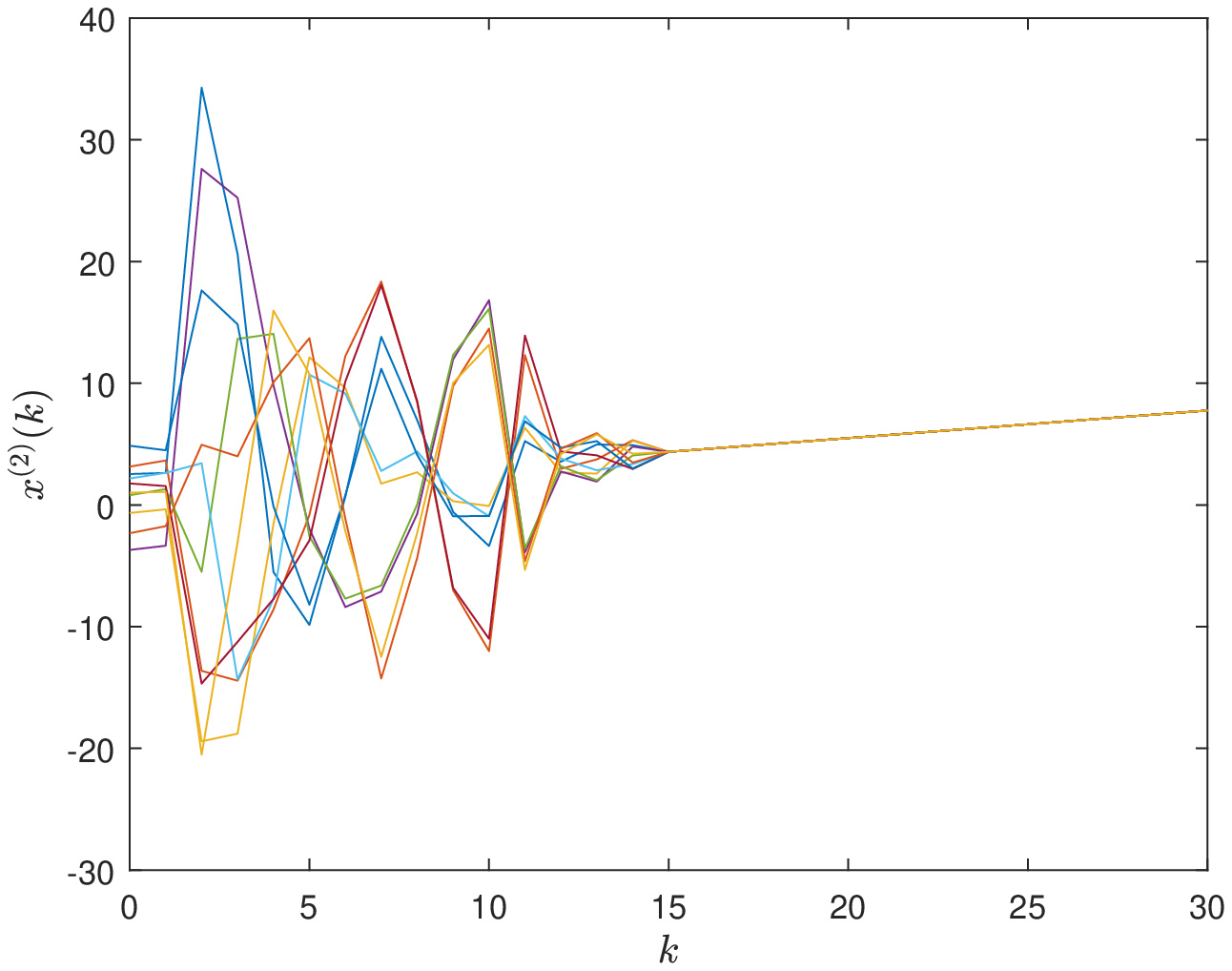}}
			\end{minipage}
		}
		\subfigure[\scriptsize Acceleration]{
			\begin{minipage}[t]{0.33\linewidth}
				\label{fig1:3}
				\centering
				\centerline{\includegraphics[width=5.4cm]{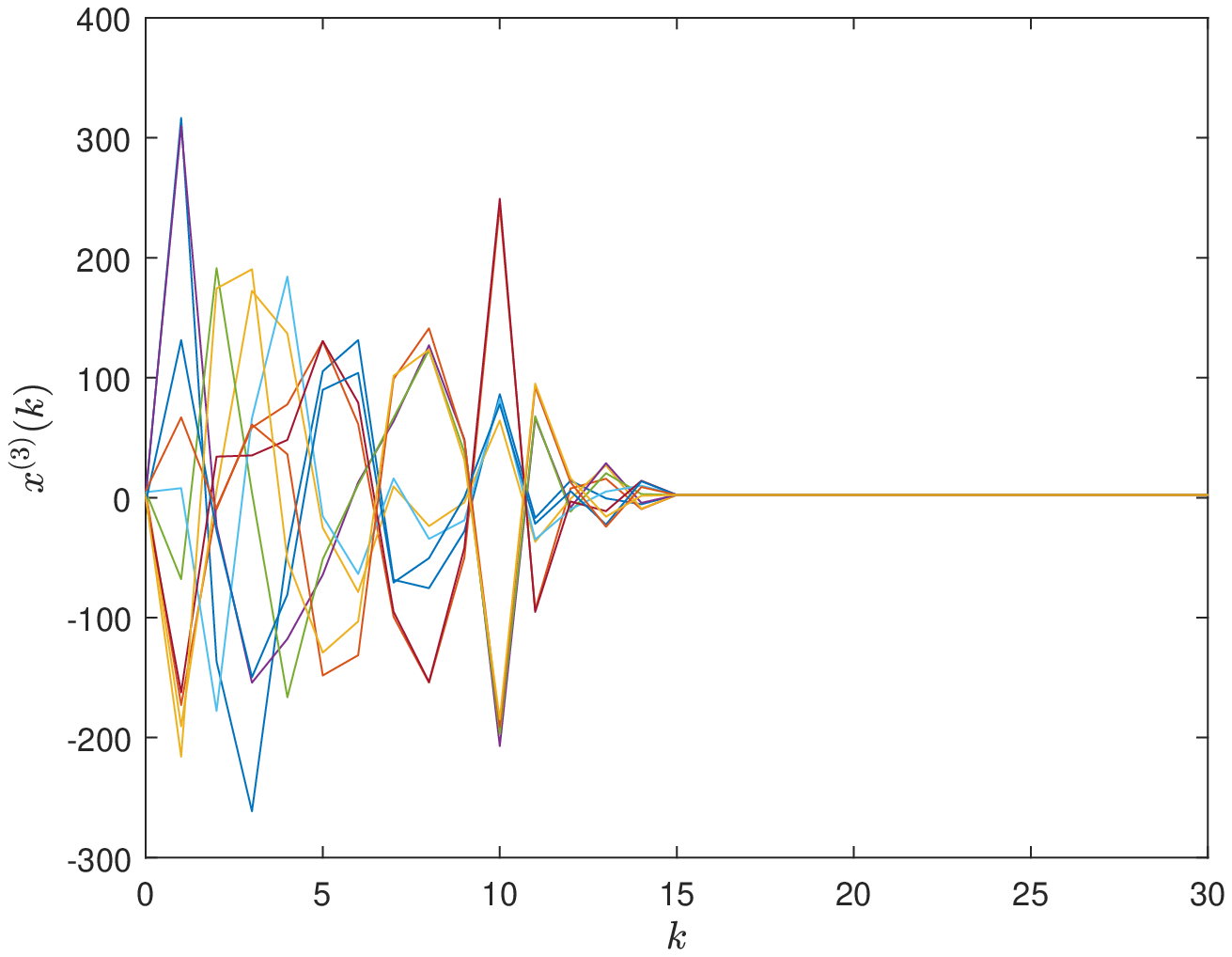}}
			\end{minipage}
		}
		\caption{Finite-time consensus of the third-order MAS}
		\label{}
	\end{figure*}



\section{Conclusions}

The problem of accelerated asymptotic and finite-time consensus of  discrete-time high-order MASs has been studied in this paper.
Firstly, a protocol with constant control gains has been introduced to achieve consensus asymptotically.
The fast consensus problem has been transformed into an optimization problem of convergence rate, and an accelerated consensus algorithm
based on gradient descent has been designed to optimize the convergence rate.
By using the Routh-Hurwitz stability criterion, the lower bound on the convergence rate has been derived, and the necessary condition to achieve this convergence rate has been proposed.
Due to the limitation of constant control, consensus can't be achieved in finite time.
Hence, a protocol with time-varying control gains has been designed to achieve the finite-time consensus.
Explicit formulas for the time-varying control gains and the final consensus state have been given.
Numerical examples have demonstrated the validity and correctness of these results.

\ifCLASSOPTIONcaptionsoff
  \newpage
\fi

\bibliographystyle{IEEEtran}
\bibliography{IEEEabrv,IEEEexample}

\end{document}